\documentclass{amsart}
\usepackage[mathscr]{eucal}
\usepackage{graphicx}
\usepackage{xypic}
\usepackage{amsfonts}
\usepackage{color}
\usepackage{amsmath}
\usepackage{amsthm}
\usepackage{amssymb}
\usepackage{latexsym}
\newfont{\msam}{msam10}

\newtheorem{theorem}[]{Theorem}
\newtheorem{proposition}[]{Proposition}
\newtheorem{corollary}[]{Corollary}
\newtheorem{lemma}[]{Lemma}
\theoremstyle{definition}
\newtheorem{definition}[]{Definition}
\newtheorem{remark}[]{Remark}

\let\nc\newcommand

\nc{\la}{\label}

\def\bthm{\begin{theorem}}
\def\ethm{\end{theorem}}
\def\blemma{\begin{lemma}}
\def\elemma{\end{lemma}}
\def\bproof{\begin{proof}}
\def\eproof{\end{proof}}
\def\bprop{\begin{proposition}}
\def\eprop{\end{proposition}}
\def\bcor{\begin{corollary}}
\def\ecor{\end{corollary}}

\def\Z{\mathbb{Z}}

\def\c{\mathbb{C}}
\def\p{\mathbb{P}}

\nc{\Hom}{{\rm{Hom}}}
\nc{\Ext}{{\rm{Ext}}}
\nc{\HOM}{\underline{\rm{Hom}}}
\nc{\EXT}{\underline{\rm{Ext}}}
\nc{\TOR}{\underline{\rm{Tor}}}
\nc{\End}{{\rm{End}}}
\nc{\GL}{{\rm{GL}}}
\nc{\PGL}{{\rm{PGL}}}
\nc{\SL}{{\rm{SL}}}
\nc{\PSL}{{\rm{PSL}}}
\nc{\Rep}{{\rm{Rep}}}
\nc{\ad}{{\rm{ad}}}
\nc{\Ad}{{\rm{Ad}}}
\nc{\dlim}{\varinjlim}

\newcommand{\Aut}{{\rm{Aut}}}
\newcommand{\Bir}{{\rm{Bir}}}
\newcommand{\id}{{\rm{Id}}}

\begin{document}
\title[]{The group of unimodular automorphisms of $\c^2$ is hopfian}
%
%

\author{Alimjon Eshmatov}
\address{Department of Mathematics, University of Western Ontario, London, Ontario, N6A 5B7, Canada}
\email{aeshmato@uwo.ca}
\author{Farkhod Eshmatov}
\address{Department of Mathematics Indiana University, Bloomington, IN 47405, USA}
\curraddr{School of Mathematics, Sichuan University, Chengdu 610064, China}
\email{faeshmat@indiana.edu)}
%
\begin{abstract}
Let $G$ be the group of unimodular automorphisms of $\mathbb C^2$.
In the paper we prove two interesting results about this group.
The first one is about absence of non-trivial finite-dimensional representations of $G$.
The second one, we show that any non-trivial group endomorphism of $G$ is a monomorphism,
which implies that $G$ is hopfian.
\end{abstract}
\maketitle
\section{introduction}
Let $\Aut (\c^2)$ be the group of polynomial automorphisms of the complex plane.
Let $G$ be the subgroup of automorphisms with Jacobian equal to $1$.
It is known that $G$ can be written as the amalgamated product
\begin{equation}
\label{Aut}
G = A *_U B\ ,
\end{equation}
where $A$ is  the subgroup of symplectic affine automorphisms, $B$ is  the Jonqui\`eres subgroup
$$
A=\{ (ax+by+e, cx+dy+f)\} \, , \quad a, \ldots \, , f \in \c , \,  ad-bc = 1
$$
$$
B=\{ (ax+q(y), a^{-1}y+f)\} \, , \quad a \in \c^* \, , \, f \in \c \, , \, q(y) \in \c[y]
$$
and  $U= A \cap B$.

In \cite{Sh}, I. R. Shafarevich proved that $G$ is simple as an infinite dimensional algebraic group. 
However, V. Danilov showed that $G$ is not simple as an abstract group. In particular, he showed that there is
 an element of the \textit{algebraic length} $26$  (w.r.t. the above amalgamated product, see Section $2.1$ 
 for the precise definition), 
 whose normal closure is not equal to $G$.  Based on the work of Danilov, 
 J.-P. Furter and S. Lamy  \cite{FL}  showed that the  normal closure of any element is non-trivial only if 
 its length is at  least $14$ and equals to  $G$ if length is less or equal than $8$. 
 This is a main observation we use to show
\begin{theorem}
\la{nonexis}
There is no non-trivial finite dimensional representation of $G$.
\end{theorem}

The group $\Aut (\c^2)$ can be naturally embedded into $\Bir(\p^2)$, and so does
$G$. Recently, J.D\'eserti \cite{De} has shown that any endomorphism of
$\Bir(\p^2)$ is injective.
 However this property is not functorial, therefore one can ask whether  
 $G$ is hopfian. Our main theorem 
\begin{theorem}
\la{hopf}
Any non-trivial endomorphism of $G$ is injective. In particular
$G$ is \textit{hopfian}, i.e. any epimorphism of $G$ is an automorphism.
\end{theorem}
\begin{remark}
In \cite{W}, D.Wright proved that $\Bir(\p^2)$ can be presented  as an amalgamation of three
subgroups $A_i$ ($i=1,2,3$) along pairwise intersections. Moreover, 
$G$ can be embedded into  $\Bir(\p^2)$ via inclusions of $A$ and $B$ into
$A_2$ and $A_3$ respectively (see, \textit{loc.cit}, Theorem~$3.13$ and Theorem~$4.21$).
In light of this, it would be interesting to see the relation between our
 Theorem~\ref{hopf} and D\'eserti's result.
\end{remark}

The paper is organized as follows.
In section $2$ we recall some facts and results needed in later sections. 
In section $3$ we prove Theorem \ref{nonexis} . In section $4$
we prove Theorem \ref{hopf}. 
\section*{acknowledgement}  
We are grateful to Yu.~Berest for interesting suggestions, questions and comments. We also thank  J.~D\'eserti and S.~Lamy for answering our questions and suggestions.
We would like to thank I.~Dolgachev and V.~Lunts for interesting discussions and comments.
 F.E is grateful to Mathematics Department of Indiana University (Bloomington) for their hospitality and support during the period when this work was carried out. 

\section{preliminaries}
\subsection{On subgroups of $G$}

By \cite{FM}, the elements of $ G  $ can be divided into two separate classes
according to their dynamical properties as automorphisms of $ \c^2 $: every
$ g \in G $ is conjugate to
either an element of $B$ or a composition of generalized H\'enon automorphisms of the form:
$$
\sigma\, g \, \sigma^{-1} = g_1 \, g_2\, \ldots \, g_m\ ,
$$
where $\, g_i = (y, x  + q_i(y)) \,$ with polynomials $\, q_i(y) \in \c[y] \,$ of degree $\, \ge  2 \,$. We say  that $ g $ is the \textit{elementary} or \textit{H\'enon} type, respectively. A subgroup $ H \subseteq G $ is called the {\it elementary} if each element of $ H $ is of elementary type.


The following results are proved in \cite[Theorem 2.4, Proposition 4.8]{L}
\begin{theorem}
\la{L}
$(a)$ Let $H$ be an elementary subgroup of $G$.
Then one of the following occurs :\\
$(1)$ $H$ is either conjugate to $A$ or $B$. \\
$(2)$ $H$ is not conjugate to $A$ or $B$. Then $H$ is abelian.\\
$(b)$ Let $g\in G$ be an element of H\'enon type. Then its centralizer is isomorphic
to $\mathbb Z \rtimes \mathbb{Z}/p\mathbb{Z}$. In particular it is countable.
\end{theorem}
\begin{remark}
It is easy to see that the centralizer of any automorphism of the elementary type
is uncountable.
\end{remark}

Let $g \in G$ be an element which is not in $U$. Then we say that it has the \textit{ algebraic length} $m$,
if $m$ is the least integer such that $\,g = \sigma_1 \cdots \sigma_m$
and each $g_i$ is either in $A$ or in $B$. We denote $|g|=m$. If
$g\in U$ then we define $|g|=0$. The following is proved in \cite[Theorem 1]{FL}
\begin{theorem}
\la{FL}
If $g \in G$ satisfies $|g| \le 8$ and $g \neq \id$, then the normal
subgroup generated by $g$ is $G$.
\end{theorem}
\subsection{Divisible groups}
We review some facts and notions about divisible groups.

We recall that an abelian group $H$ is \textit{divisible} if for each $g \in H$ and positive integer $n$ there is an element $h \in H$
with $g= n \, h$ .

Finite abelian groups are not divisible. Among familiar infinite abelian groups,
$\mathbb{Q}, \mathbb{R}, \c, \c[y],  \c^* $  are divisible but $\mathbb{R}^* $ and $\Z$ are not.

The following fact is useful
\begin{lemma}
A quotient of a divisible group is divisible.
\end{lemma}
In particular we have
\begin{corollary}
There are no non-trivial homomorphisms from divisible groups to finite groups.
\end{corollary}

\subsection{Solvable subgroups of $\GL_n (\c)$}
 We recall a classical
characterization of solvable subgroups of $\GL_n (\c) $ due to A.~I.~Maltsev.
Maltsev's theorem is a generalization of the Lie-Kolchin theorem, it gives a description 
of all solvable subgroups of $\GL_n(\c)$
for its proof we refer to \cite[Theorem 3.1.6]{LR}.
\begin{theorem}[Maltsev]
\la{malcev}
Let $\Gamma $ be any solvable subgroup of $\GL_n (\c) $.
Then $ \Gamma $ has a finite index normal subgroup which is
conjugate to a subgroup of upper triangular matrices.
\end{theorem}

Let $U_n$ be the group of upper triangular matrices with entries $1$ in the diagonal.
One can easily show
\begin{lemma}
\la{nilp}
$U_n$ is a subgroup of $\GL_n (\c)$ of nilpotency class $n-1$.
\end{lemma}
\section{On nonexistence of finite-dimensional representations of $G$}
Let $\,\rho : G \to \GL_n (\c)$ be a group homomorphism.
Since $\,G = [G,G]$ we can easily see that $\,\rho (G)\subseteq \SL_n (\c)$.
Now we prove
\begin{proposition}
\la{inj}
If $\rho$ is non-trivial then $\rho|_A$ and $\rho|_B$ must be injective.
\end{proposition}
\begin{proof}
Suppose the kernel of $\rho$ contains $g$, a non-trivial element of $A$ or $B$.
Since $g$ is of length at most one, by Theorem \ref{FL} the normal
closure of any such element is equal to $G$. This implies that $\rho$ must be trivial.
\end{proof}
Since $B$ is a solvable group, $\rho(B)$ is a solvable subgroup of $\SL_n (\c)$. The following lemma
gives more precise description of $\rho(B)$:
\begin{lemma}
\la{eig}
All eigenvalues of $\rho (x + \lambda y^k, y)$ and $\rho(x, y + \mu x^k)$ are $1$ for all
$\lambda \, , \mu \in \c$ and $k \geq 0$.
\end{lemma}
\begin{proof}
We set $\, A_{k,\lambda}:=\rho(x + \lambda y^k, y)$ and $\, B_{k,\mu}:=\rho(x, y + \mu x^k)$.
We will prove the lemma for $\, A_{k,\lambda}$ since a proof of $\, B_{k,\mu}$ is analogous.
We have
$$
 (\nu^{-1} x, \nu y) \, \circ \,(x + \lambda y^k, y) \, \circ \, ( \nu x, \nu^{-1} y) \, =
 \, (x+\lambda \nu^{k+1} y^k, y) \, .
$$
for some $\nu \in \c^*$. This means a matrix $A_{k,\lambda}$ is similar to $A_{k,\lambda \nu^{k+1}}$. In particular  for $\nu^{k+1}\in \mathbb{Z} $ we obtain that  $A_{k,\lambda}$ is similar to any of its power
$A_{k, m \lambda} = A^m _{k,  \lambda}$. If $\,\{a_1, \ldots , a_n\,\}$ is the set of
eigenvalues of $A_{k,\lambda}$, then this set equal to the set $\,\{a_1^m, \ldots , a_n^m\,\}$
for any $\,m\ge 1$. This implies that $\,a_1^{m_1}=1, \ldots , a_n^{m_n}=1$ for some positive
$m_1, m_2, \ldots , m_n$. Finally, choosing $\,m=m_1 \, m_2 \, \ldots \,  m_n\,$ we have
$\,\{a_1^m= \ldots = a_n^m=1\,\}$. Hence $a_1 = a_2 = \ldots = a_n =1$.
\end{proof}
Consider the unitriangular subgroup $B_0 \subseteq B$ consisting of elements
$$
(x+ p(y) , \, y + f) \, .
$$
Then we have
\begin{proposition}
\la{u_n}
$\rho (B_0)$ is conjugate to a subgroup of $U_n$.
\end{proposition}
\begin{proof}
First we note $B_0$ is a solvable subgroup $B$ and $\rho(B_0)$ is a solvable subgroup of $\SL_n (\c)$.
Therefore by Theorem \ref{malcev} it has a normal triangularizable subgroup $T$ which has a finite index in $\rho(B_0)$.
In other words $\rho(B_0) / T$ is a finite group. A surjective homomorphism $B_0 \rightarrow \rho(B_0) / T$
 induces a homomorphism $[B_0, B_0] \rightarrow \rho(B_0) / T$. Since a group $[B_0, B_0] \cong \c[y]$ is divisible, 
 $[B_0, B_0] \rightarrow \rho(B_0) / T$ must be trivial.
Therefore we have a surjective homomorphism $B_0 / [B_0,B_0] \rightarrow \rho(B_0) / T$. However 
$B_0 / [B_0,B_0] \cong \c$ is divisible therefore
a group $\rho(B_0) / T$ must be trivial. Hence $\rho(B_0)$ is conjugate to a subgroup of upper triangular matrices. 
Now by Lemma \ref{eig} $\rho(B_0)$ is also unipotent.
\end{proof}
On the other hand
\begin{proposition}
\la{nonilp}
 $B_0$ is not a nilpotent group.
\end{proposition}
\begin{proof}
One can compute that  the group $B^{(1)}_0 = \, [B_0, B_0 ]$ consists of elements
$$
(x+ p(y), y) \quad \mbox{for all} \quad p(y) \in \c[y]
$$
On the other hand
$$
B^{(2)}_0 = [B_0 , B^{(1)}_0 ] = B^{(1)}_0
$$
So it stabilizes $ 1 \neq B^{(1)}_0 = B^{(2)}_0 = \ldots $. Hence it is not nilpotent.
\end{proof}
\begin{proof}[Proof of Theorem \ref{nonexis}]
Suppose there is a non-trivial homomorphism $\rho : \, G \rightarrow \GL_n (\c)$. By  Proposition \ref{inj} its restriction to $B_0$ must be injective. From Proposition \ref{u_n}
it follows that $\rho (B_0)$ can be conjugated to a subgroup of $U_n$ and hence is nilpotent. This  contradicts to Proposition \ref{nonilp}.
\end{proof}

There are some interesting consequences of this result which are of independent interest.
Let $\mathtt{Cr}(n)$  be the Cremona group of birational automorphisms of $\p^n$.
Then the above result implies
\begin{corollary}
$(a)$ There is no non-trivial finite dimensional representation of $\mathtt{Cr}(2)$. \\
$(b)$ $\Aut(\c^n)$  and $\mathtt{Cr}(n)$ are not linear, i.e., these groups have
no faithful representations in $\GL_n(\c)$.
\end{corollary}
\begin{proof}
$(a)$ Follows from the fact that the subgroup $\SL_2(\c)$ in $G$ is also a subgroup of $\PGL_3(\c)$
in $\mathtt{Cr}(2)$.\\
$(b)$ It follows immediately from the fact that $G$ is a subgroup of both $\Aut(\c^n)$ and $\mathtt{Cr}(n)$.
\end{proof}
Results of this corollary for $\mathtt{Cr}(n)$  were proved earlier by D.Cerveau and J.D\'eserti \cite{CD}.
\section{Endomorphisms of the group $G$}
For $g \in G$ we denote by $\Ad_g$  the inner automorphism of $G$
given by $g \, ( \cdot ) \, g^{-1}$. To prove our theorem it
suffices to show: given a non-trivial $\phi : G \rightarrow G$ homomorphism
there are $g , \, h  \in G$ such that composition $\Ad_g \,\circ\, \phi
\, \circ\,\Ad_h$ is a monomorphism. First we will show that any
non-trivial endomorphism of $G$ can be composed by an inner
automorphism to give an endomorphism which induces injective
endomorphisms of its subgroups $A$ and $B$, namely $\phi(A) \subset
A$ and $\phi(B) \subset B$. Following \cite{FL} one can define
systems of representatives of the non-trivial left cosets $A/U$ and
$B/U$ by 
$$
 I = \{ \, (\lambda x + y, -x) \, , \, \lambda \in \c \, \}
$$ 
$$
J= \{ \,  (x+ p(y),y) \, , \, p(y) \in y^2 \c[y] \backslash \{0\} \, \}
$$
respectively. We can prove
\begin{proposition}
\la{auxi}
 Let $\mu = (x+p(y),y)$ such that $\deg(p)=n \geq 2$. Then $A \cap \mu A \mu^{-1}$
is a subgroup of $H$ defined as
\begin{equation}\label{eq2}
 H= \{(x+by+e, y) \, | \, b, \, e \in \c \} \rtimes \Z_{n+1}  
\end{equation}
where $\Z_{n+1}$ is the cyclic subgroup of $(\lambda x, \lambda^{-1} y)\,
, \,  \lambda \in \c^*$.
\end{proposition}
\begin{proof}
We will consider two cases: $g\in A \backslash U$ and $g \in U$. In the first case
$\mu \, g \, \mu^{-1}$ is a word of length $3$ so it can not be in $A$. If $g \in U$
we have
$$
\mu \, g \, \mu^{-1} = \, (\lambda x + \lambda p(y)-p(\lambda^{-1}y+f)+by+e,\, \lambda^{-1}y+f)
$$
where $g=(\lambda x +by+e,\, \lambda^{-1}y+f)$. The element $\mu \, g \, \mu^{-1} $
belongs to $A$ if and only if $\deg(\lambda p(y)-p(\lambda^{-1}y+f)) \leq 1$ which can only happen
if $\lambda^{n+1}=1$. This immeadiately imply the statement.
\end{proof}
\begin{proposition}
\la{int}
Let $ \phi: G \to G$ be a non-trivial group homomorphism. Then
\begin{enumerate}
\item[(a)] Restrictions of $\phi$ to $A$ and $B$ are group monomorphisms.
\item[(b)] $\phi(A) \cap \phi(B) =\phi(U)$.
\end{enumerate}
\end{proposition}
\begin{proof}\mbox{}
\begin{enumerate}
\item[(a)]
Let $a \in A \cup B$ be an element $a\neq 1 $ such that  $\phi(a)=1$.
Then by Theorem~\ref{FL} we have $\phi(G)=1$, which is impossible.

\item[(b)] It is clear that $\phi(U) \subset \phi(A) \cap \phi(B)$.
Now if $\phi(a)=\phi(b)$ for some $a\in A$ and $b\in B$ then $\phi(ab^{-1})=1$.
Again by Theorem~\ref{FL} implies that $\phi$ is injective on words of length 2.
Hence $a=b \in U$.
\end{enumerate}
\end{proof}
\begin{theorem}
\la{main}
Let $\, \phi: G \to G$ be a non-trivial group homomorphism.
Then composing $\phi$ by proper inner automorphisms of $G$, we obtain a homomorphism
$\tilde \psi$ such that
$$ \tilde \psi(A) \, \subset \, A \, , \,  \tilde \psi(B) \, \subset \, B \, ,\,
\tilde \psi(U) \, \subset \, U\,.$$
Moreover $\psi$ restricted to  $A, B$ and $U$ gives injective endomorphisms.
\end{theorem}
\begin{proof}
By Theorem~\ref{L}(b) each element of subgroups
$\phi(A)$ and $\phi(B)$ is elementary. Hence, by part (a)
of the same theorem both $\phi(A)$ and $\phi(B)$
can be conjugated to either $A$ or $B$. The subgroup
$\phi(A)$ can not be conjugated to $B$, since $B$ is solvable
while $\phi(A)$, being isomorphic to $A$, is not. So
$ \phi(A) \, \subset \, \sigma A \sigma^{-1}$ for some $\sigma \in G$.
Composing $\phi$ by $\Ad_{\sigma^{-1}}$ we can assume that
$\phi(A) \subset A$.

For  $\phi(B)$ we have that it is conjugate to $A$ or $B$.
We now discuss each case.

Case $1$. Assume that $\phi(B) \subset \mu A \mu^{-1}$ for some $\mu \in G$.
Let $\mu$ be of length 1. If $\mu \in A$ then $\phi(B) \subset A$ and this implies $\phi(G) \subset A$.
Taking projection of $A$ onto $\SL_2(\c)$ we get a representation of $G$
which by Theorem~\ref{nonexis} is trivial. Therefore we obtain a homomorphism 
$G \rightarrow \mathcal{T}$, where $\mathcal{T}=\{(x+e,y+f) \, | \, e, f \in \c \}$ the translation subgroup,  which must be injective when restricted to $A$ and $B$ by Proposition
\ref{int}. This is impossible,
hence $\mu \notin A$.

Now assume that $\mu \in B$. Without loss of generality we can assume $\mu$ is a non-trivial representative in $B/U$, with $\mu =(x+p(y),y)$ with $p \in y^2 \c[y] \, \backslash \, 0$.
Then by Proposition~\ref{int}(b) we have $\phi(U) \subseteq A \cap \mu A \mu^{-1}$.
Then according to Proposition \ref{auxi} the group $U$ embeds into  \eqref{eq2}.
Note that $U$ contains a cyclic group of any finite order, which contradicts to the last embedding.
Hence $\mu \notin B$.

Let
\begin{equation}\label{eq3}
\mu = w_0 \,w_1 \, \ldots \, w_n \,  , \quad n\geq 2
\end{equation}
be a reduced word of length $n$ in $G$ where $w_0 \in U$ and $w_i$ for $i>0$ are
in $I$ or $J$. Without loss of generality we can assume $w_n \in J$.
Once again by Proposition~\ref{int}(b)  we must have $\phi(U) \subseteq A \cap \mu A \mu^{-1}$.
Then
$$\mu \, a \, \mu^{-1}= w_0 \,w_1\, \ldots \,w_n \, a \, w_n^{-1}\, \ldots
\, w_1^{-1} w_0 ^{-1} \, ,$$
and $\mu a \mu^{-1} \in A$ if and only if
$$
\nu \, = \, w_0^{-1}\, \mu \, a \, \mu^{-1} \, w_0 \, = \,  w_1\,...
\,w_n \, a \,w_n^{-1}\,...\, w_1^{-1} \, .$$
is in $A$. Now either  $a\in A \backslash U$ or $a \in U$. In the first case
$\nu$ is a word of length $2n+1$ so it can not be in $A$. If $a \in U$
then  the element $w_n a w_n^{-1}$ is in $U$ or $B\backslash U$.  
In the latter case $\nu$ is at least of
 length $2n-1>2$ and therefore it is not in $A$. If $w_n a w_n^{-1}$ is in $U$ then all such elements are in $A \cap w_n A w_n^{-1}$, i.e. $A \cap \mu A \mu^{-1}$ can be embedded in $A \cap w_n A w_n^{-1}$. By Proposition \ref{auxi} $A \cap w_n A w_n^{-1}$ is a subgroup of $H$ 
and hence $A \cap \mu A \mu^{-1}$ can be embedded into H
By injectivity of $\phi$ on $U$, this is impossible since $U$ contains a cyclic group of any finite order. Thus $\phi(B)$ can not be conjugated to a subgroup of $A$.

 Case $2$. Let $\phi(B) \subset \mu B \mu^{-1}$ for some $\mu \in G$.
 Then $\phi(U) \subset A \cap \mu B \mu^{-1}$. Now if $\mu$ is in $A$ or $B$
 then we done. So we assume that $\mu$ has a reduced form as in \eqref{eq3}.
  Arguing as above we can show that $U$ can not be isomorphic to a subgroup of
  $ A \cap \mu B \mu^{-1}$.

Thus summarizing all cases we conclude that by composing $\phi$ by an inner auto- morphism of A if necessary, we obtain a homomorphism $\tilde \psi$̃ with properties stated in the theorem.
\end{proof}
We can slightly refine the previous theorem
\begin{lemma}
\la{Z2} 
Let $\tilde \psi$ be as in Theorem \ref{main}. Then composing $\tilde \psi$ by inner
automorphisms $\Ad_g$ with $g \in U$ we obtain $\psi$ such that
$$
\psi(A) \, \subset \, A \, , \,  \psi(B) \, \subset \, B \, ,\,
\psi(U) \, \subset \, U\,
$$
and
$$
\psi ((-x,-y))=(-x,-y)
$$
\end{lemma}
\begin{proof}
By Theorem \ref{main} $\tilde \psi(U) \, \subset \, U$ therefore
$$
\tilde \psi ((-x,-y))=(\lambda \, x+cy+e,\lambda^{-1} y+f) \quad \mbox{for some} \quad \lambda \in \c^* \, , \, c,e,f \in \c
$$
Since $(-x,-y)$ is of order 2 and $\psi$ is injective on $U$ we have $\tilde \psi ((-x,-y))=(-x+e,-y+f)$. Now if we take
$g=(\frac{e}{2} , \frac{f}{2})$, the composition $\psi= \Ad_g \circ \tilde \psi$ gives us desired homomorphism.
\end{proof}
In particular we have
\begin{corollary}
\la{SL2}
 Let $\psi $ be as in Lemma~\ref{Z2}. Then $\psi (\SL_2 (\c)) \subset \SL_2 (\c)$.
\end{corollary}
\begin{proof}
Let $Z_A (g)$ be the centralizer subgroup of $g$ in $A$. Then $\psi(Z_A (g)) \subseteq Z_A (\psi(g))$.
Now proof follows from $Z_A ((-x,-y))=\SL_2 (\c)$. 
\end{proof}
\section{The proof of Theorem~\ref{hopf}}
We need to show that a homomorphism $\psi$ with properties
described in Theorem \ref{main} and Lemma \ref{Z2} is injective.
By Theorem \ref{main} and Lemma \ref{Z2} for $\psi$ we have induced quotient maps
$$  \bar{\psi}_{A}\, :\, A/U \, \to \, A/U \, , \quad  \bar{\psi}_{B} \,
 : \, B/U \, \to \, B/U$$
To prove our result it is sufficient to show that these maps are injective. Indeed, assume that these two maps are injective
and let $\, g \in G$ be $g \neq 1$. It has a normal form $ g\, = \, w_0 \, w_1 \, \ldots \, w_n \neq 1$ where
$w_0 \in U$ and $w_i$ are in $I$ or $J$. Then
$$
\psi(g) \, = \, \psi(w_0) \, \psi(w_1)\, \ldots \, \psi(w_n) \, ,
$$
 where $\psi(w_0) \in U$ and $\psi(w_i)$ are non-trivial representatives in $A / U$ or  $B / U$ by injectivity of $\bar{\psi}_A$ and $\bar{\psi}_B$.
So the above presentation of $\psi(g)$ is a reduced word and can not be
equal to $1$.

{\it Proof of  injectivity of $\bar{\psi}_A$ }: Recall $A/U$ consists of $gU$,
for $g \in I$. Suppose $\psi(g) \in U$ for some $g \in I$. Since $g$ and $U$ generate $A$,
we have $\psi(A) \subset U$. The latter
contradicts injectivity of $\psi |_A$ since $U$ is solvable while $A$ is not.
 Therefore $\bar{\psi}_{A}$ must be injective.

{\it Proof of  injectivity of $\bar{\psi}_B$}: Coset representatives
 $B/U$ consists of $gU$, where $g \in J$.
 Suppose that $\bar{\psi}_{B}$ is not injective, namely
there is  $g= (x+p(y),y)$ with nonzero $p(y) \in y^2 \, \c[y]$  such that
$\psi((x+p(y),y)) \in U$. Let $\deg (p(y))=n >1$. Then the image of the following is also in $U$
$$
[(x,y+1),(x+p(y),y)]=(x+p(y+1)-p(y),y) \, .
$$
Note that the degree of $p(y+1)-p(y)$ is exactly $n-1$ and taking commutator with $(x,y+1)$
lowers degree exactly by 1. Therefore taking commutator $(x+p(y),y)$ with $(x,y+1)$ exactly
$n-2$ times gives us
$$ [(x,y+1),[(x,y+1), ..., [(x,y+1),(x+p(y),y)]...]\, =\,(x+q(y),y) \, ,$$
 where $q(y)$ is a quadratic polynomial and $\, \psi((x+q(y),y)) \in U$. Therefore
$\, \psi((x+y^2,y))$ is also in $U$.  Note since $\psi(B) \subset B$ we also have $\psi(B^{(i)})  \subset B^{(i)}$
for derived series of $B$. In particular since $B^{(2)}= \{(x+p(y), y)\}$ we obtain that
$\, \psi((x+y^2,y))= (x+cy+e, y)$ for some $c, e \in \c$.
We have
$$
 [(-x,-y), (x+y^2,y)] \, =\, (x-2y^2,y) \quad \mbox{and} \quad [(-x,-y), (x+cy+e, y)] \, =\, (x-2e,y)
$$
Therefore by Lemma \ref{Z2} we have $\psi((x-2y^2,y))= (x-2e,y)$.  On the other hand,
$$
\psi((x-2y^2,y))=\psi((x+y^2,y)^{-2})=
(x+cy+e, y)^{-2}= (x-2cy-2e,y).
$$
Hence $c=0$ and therefore $\, \psi((x+y^2,y))\in \mathcal{T}$. Note then by Corollary \ref{SL2} the element
$\psi((-y,x) \, (x+y^2,y) \, (y,-x))$ is also in $\mathcal{T}$. Therefore $\psi$ maps commutator of $(x+y^2,y)$ and $(-y,x) \, (x+y^2,y) \, (y,-x)$ which is
of length $8$ to identity. This is impossible by Theorem \ref{FL}. This completes a proof of injectivity of $\psi_B$
hence of $\psi$.

\begin{remark}
One can prove using similar arguments that $\Aut (\c^2)$ is also hopfian. However one can easily observe not every
endomorphism of $\Aut (\c^2)$ is injective.
\end{remark}
\bibliographystyle{amsalpha}

\begin{thebibliography}{A}
%
%
\bibitem[CD]{CD}
D.Cerveau and J.D\'eserti,
\textit{Transformations birationnelles de petit degr\'e},  arXiv:0811.2325
%
%
\bibitem[De]{De}
J. D\'eserti,
\textit{Le groupe de Cremona est hopfien. }, C. R. Math. Acad. Sci. Paris  \textbf{344}  (2007),  no. 3, 153--156.
%
\bibitem[FL]{FL}
J.-P. Furter, S. Lamy,
\textit{Normal subgroup generated by a plane polynomial automorphism}, Transformation Groups, {\bf 15}(3), (2010), 577--610.
 %
 \bibitem[FM]{FM}
 S. Friedland, J. Milnor, \textit{Dynamical properties of plane polynomial automorphisms},  Ergodic Theory Dynam. Systems
 \textbf{9}  (1989),  no. 1, 67--99.
 %
\bibitem[L]{L}
 S. Lamy, \textit{L'alternative de Tits pour ${\rm Aut}[\c^2]$, }  J. Algebra  \textbf{239}  (2001),  no. 2, 413--437.
 %
%
\bibitem[LR]{LR}
J. C. Lennox and D. Robinson, \textit{The Theory of Infinite Soluble Groups},
Oxford University Press, Oxford, 2004.
%
\bibitem[Sh]{Sh}
 I. R. Shafarevich, \textit{On some infinite-dimensional groups. II}, Izv. Akad. Nauk SSSR Ser. Mat., {\bf 45}(1), (1981), 214--226.
%
\bibitem[W]{W}
D.~Wright, \textit{Two-dimensional Cremona groups acting on simplicial complexes}, Trans. Amer. Math. Soc., {\bf 331(1)} : 281–-300, 1992.
%
\end{thebibliography}

\end{document}